\documentclass[reqno, 11pt]{amsart}
\usepackage{amsmath}
\usepackage{dsfont}
\usepackage{amsfonts}
\usepackage{amssymb}
\usepackage{commath}
\usepackage{mathrsfs}
\usepackage{enumerate}
\usepackage{comment}
\usepackage[backend=biber, style=alphabetic, url=false]{biblatex}
\usepackage[bookmarks=false,pdfstartview=FitH, pdfborder={0 0 0}, colorlinks=true,citecolor=blue, linkcolor=blue]{hyperref}
\usepackage{aliascnt} 

\theoremstyle{plain}
\newtheorem{theorem}{Theorem}[section]

\newaliascnt{conjecture}{theorem}
\newtheorem{conjecture}[conjecture]{Conjecture}
\aliascntresetthe{conjecture}

\newaliascnt{corollary}{theorem}
\newtheorem{corollary}[corollary]{Corollary}
\aliascntresetthe{corollary}

\newaliascnt{lemma}{theorem}
\newtheorem{lemma}[lemma]{Lemma}
\aliascntresetthe{lemma}

\newaliascnt{fact}{theorem}

\aliascntresetthe{fact}

\newaliascnt{claim}{theorem}

\aliascntresetthe{claim}

\newaliascnt{proposition}{theorem}

\aliascntresetthe{proposition}

\theoremstyle{definition}

\newaliascnt{definition}{theorem}
\newtheorem{definition}[definition]{Definition}
\aliascntresetthe{definition}

\newaliascnt{example}{theorem}

\aliascntresetthe{example}

\theoremstyle{remark}

\newaliascnt{remark}{theorem}
\newtheorem{remark}[remark]{Remark}
\aliascntresetthe{remark}

\newcommand{\sca}{\mathrm{Scal}}
\begin{document}
\title[Sectional-scalar Pinching]{Sharp Weitzenb\"{o}ck and PIC2 Estimates from Sectional-Scalar Curvature Pinching}
\author[J.~Ge]{Jian Ge}
\address[Ge]{School of Mathematical Sciences, Laboratory of Mathematics and Complex Systems, Beijing Normal University, Beijing 100875, P. R. China.}
\email{jge@bnu.edu.cn}
\thanks{Partially supported by NSFC 12371049 and the Fundamental Research Funds for the Central Universities}
\subjclass[2020]{Primary: 53C20; Secondary: 53C21, 53C24}
\keywords{sectional-scalar curvature pinching, Weitzenböck curvature operator, harmonic two-forms, Betti numbers, rigidity, Fubini--Study metric}
\begin{abstract}
	Let $V$ be an $n$-dimensional Euclidean vector space, ,where $n\ge 4$, and $\ell = \lfloor\frac{n}{2}\rfloor$. We prove the sharp pointwise estimate
	\[
		q_2(E) \ge -\frac{2(\ell -1)}{3\ell} \sca(E) \mathrm{Id}_{\Lambda^2V^*}
	\]
	for every algebraic curvature tensor $E$ on $V$ with nonnegative sectional curvature. Applying this estimate to the decomposition $\operatorname{Rm}_{g}=K_{\min}I+E$, we obtain the vanishing of $H^2(M; \mathbb{R})$ under a dimension-dependent strict sectional-scalar curvature pinching condition. At the weak endpoint, all harmonic two-forms are parallel. Apart from the flat case, this yields $b_2(M)=0$ in odd dimensions and $b_2(M)\le 1$ in even dimensions. At even-dimensional endpoint, $b_2(M)>0$ forces $(M, g)$ to be isometric, up to scaling, to $\mathbb{CP}^{\ell}$ with its Fubini-Study metric. As a consequence every closed five-dimensional manifold satisfying the strict pinching condition implies is a rational homology sphere. An anisotropic rescaling of the same homogeneous four-frame estimate also gives the sharp pointwise sectional-scalar pinching criterion
	\[
		K_{\min} \ge \frac{n(n-1)}{n^2-n+12}S_0 \quad \Longrightarrow \mathrm{PIC2}.
	\]
	The strict pinching places the curvature tensor in the interior of PIC2 and normalized Ricci flow brings it to a positive constant sectional curvature.
\end{abstract}
\maketitle

\section{Introduction and Main Results}

A central theme in Riemannian geometry is the interplay between topology and curvature. The classical Berger--Klingenberg, cf \cite{Ber1960a}, \cite{Kli1961}, sphere theorem states that a complete, simply connected Riemannian manifold whose sectional curvature satisfy $1<K\le 4$ is homeomorphic to a sphere. The differential sphere theorem of Brendle and Schoen \cite{BS2009a}and its weak-endpoint classification \cite{BS2009a} replace this, under pointwise quarter-pinching, by a spherical space form or a compact rank one symmetric space. Related rigidity, second Betti number, and finiteness results under positive pinching include \cite{GG1987}, \cite{FR1999}, \cite{PT1999}. This strict versus endpoint patten motivates the present sectional-scalar pinching problem.

Let $(M^{n}, g)$ be a Riemannian manifold. At $x\in M$, set
\begin{equation*}
	K_{\min}(x)=\min_{\sigma \subset T_{x}M} K(\sigma),\qquad 
	S_{0}(x)= \frac{\sca(M^{n})|_{x}}{n(n-1)}
\end{equation*}
Therefore $S_{0}$ is the average of sectional curvatures at a point. Yau asked whether the sphere theorem has an analogue with the maximum sectional curvature replaced by normalized scalar curvature \cite{Yau1993}. The standard sharp formulation is the following:
\begin{conjecture}[Yau's pinching problem]
	Let $(M^{n}, g)$ be a closed, simply connected Riemannian manifold of dimension $n\ge 4$. If
	\begin{equation*}
		K_{\min} >\frac{n-1}{n+2} S_{0},
	\end{equation*}
	pointwise, then $M^{n}$ is homeomorphic to the $n$-sphere $\mathbb{S}^{n}$.
\end{conjecture}

For even $n$, the coefficient is sharp: the Fubini-Study metric can be normalized to have $K_{\min}=1$, $K_{\max}=4$ and $\sca=n(n+2)$. Hence
\begin{equation*}
	K_{\min}=\frac{n-1}{n+2}S_{0}.
\end{equation*}
so $\mathbb{CP}^{n/2}$ lies exactly at the endpoint. Gu and Xu proved the sphere conclusion under the stronger coefficient $n(n-1)/(n^{2}-n+6)$, \cite{GX2012}. Li recently improved that coefficient to $n(n-1)/(n^{2}-n+12)$, cf. \cite{Li2026}. In dimension $4$, Yau's conjecture was first proved by Costa and Ribeiro Jr. under a weaker curvature assumption \cite{CR2014}. Li's strict coefficient agrees with Yau's hence recover the dimension $4$ case. For every dimension $n\ge 5$, Li's coefficient is larger than Yau's.

Our argument belongs to the Bochner tradition for harmonic two-forms, beginning with the classical curvature term estimate of Gallot and Meyer \cite{GM1975}. Micallef and Wang obtained strong two-form consequences from nonnegative isotropic curvature \cite{MW1993}, see also Seaman \cite{Sea1993}; modern algebraic treatments of Weitzenb\"{o}ck curvature terms include \cite{Lab2015}, \cite{BM2022}, \cite{PW2021}, \cite{NPW2023}. Wan assumed an inequality involving Ricci and sectional curvatures \cite{Wan2014}. Ni and Wilking's generalized Berger inequality for pinched flag curvature supplies the four-frame estimate used below, \cite{NW2010}. Their flag-curvature pinching is not an additional hypothesis of our theorem; only the pointwise algebraic estimate extracted from it is used.

Motivated by Yau's conjecture, we consider the dimension-dependent constant:
\[
	\beta_n=
	\begin{cases}
		\dfrac{n-1}{n+2},&n\ \text{even},\\[10pt]
		\dfrac{n(n-3)}{n^2-6},&n\ \text{odd}.
	\end{cases}
\]
For even $n$ this is Yau's coefficient; for odd $n$, it is strictly smaller than Yau's coefficient $(n-1)/(n+2)$. Our strict result is a cohomological conclusion rather than a sphere theorem.

\begin{theorem}[Strict pinching]\label{thm:Main}
	Let $(M^{n}, g)$ be a closed Riemannian manifold of dimension $n\ge 4$. If the metric satisfies:
	\begin{equation*}
		K_{\min}> \beta_{n}S_{0},
	\end{equation*}
	then
	\begin{equation*}
		H^{2}(M; \mathbb{R})=H^{n-2}(M; \mathbb{R})=0.
	\end{equation*}
\end{theorem}

The strict inequality in \autoref{thm:Main} is essential. The endpoint retains substantially more information than vanishing along.

\begin{theorem}[Weak endpoint]\label{thm:WeakEnd}
	 Let $(M^{n}, g)$ be a closed connected Riemannian manifold of dimension $n\ge 4$. Suppose pointwise that
	 \begin{equation*}
	 	K_{\min} \ge \beta_{n}S_{0}.
	 \end{equation*}
	 Then either $g$ is flat, or 
	 \begin{equation*}
		 b_{2}(M)=0\quad \text{if}\ n\ \text{is odd},\qquad b_{2}(M)\le 1\quad \text{if}\ n\ \text{is even}.
	 \end{equation*}
	 Moreover, in the nonflat case, if $n=2\ell$ and $b_{2}(M)>0$, then $(M, g)$ is isometric, up to scaling, to $\mathbb{CP}^{\ell}$ equipped with the Fubini-Study metric.
\end{theorem}

The flat alternative is unavoidable, because both $K_{\min}$ and $S_{0}$ vanish identically for flat metric. 

\begin{corollary}[Five-dimensional consequence]\label{cor:Berger}
	Let $(M^{5},g)$ be a closed connected Riemannian manifold. If 
	\[
		K_{\min}> \frac{10}{19}S_{0},
	\]
	then $M$ is a rational homology sphere.
\end{corollary}

Berger's classical five-dimensional theorem assumes strict $4/23$-pinching of the sectional curvature and proves that vanishing of the second Betti number; the resulting positive curvature argument gives the same rational homology sphere conclusion \cite{Ber1963}. Note that Berger's hypothesis and our sectional-scalar pinching are incomparable even though the topological conclusion are the same.

A parameter-dependent anisotropic rescaling of the homogeneous four-frame estimate of \autoref{lem:NW} will be used to prove the main sharp estimate for the $q_{2}$ term, see \autoref{thm:SharpEstimate}. On the other hand it has a second application, logically independent of the $q_{2}$ estimate. The coefficient below is exactly the scalar $K_{\min}$ coefficient already appear in Li's Theorem 1.4 \cite{Li2026}; the point of the present argument is its sharp interpretation as the threshold forcing the full PIC2 cone, including the weak endpoint.

\begin{theorem}[Sharp criterion for PIC2]\label{thm:PIC2}
	Let $R$ be an algebraic curvature tensor on an $n$-dimensional Euclidean space $V$, $n\ge 4$, and put 
	\[
		K_{\min}(R)=\min_{\sigma\in \mathrm{Gr}_{2}(V)}K_{R}(\sigma), \qquad S_{0}(R) = \frac{\sca(R)}{n(n-1)}.
	\]
	If 
	\begin{equation}\label{eq:KPIC01}
		K_{\min}(R)\ge \frac{n(n-1)}{n^{2}-n+12} S_{0}(R),
	\end{equation}
	then $R$ lies in the PIC2 cone. If the inequality in \eqref{eq:KPIC01} is strict, then $R$ lies in the interior of the PIC2 cone.
\end{theorem}

The coefficient in \autoref{thm:PIC2} is sharp among pointwise algebraic criteria of this form: For every $\gamma< \frac{n(n-1)}{n^{2}-n+12}$, there exists an algebraic curvature tensor $R$ such that $K_{\min}> \gamma S_{0}(R)$ but $R$ does not even have nonnegative isotropic curvature.

\begin{corollary}\label{cor:RicciFlow}
	Let $(M^{n}, g)$ be a closed connected Riemannian manifold of dimension $n\ge 4$. If pointwise on $M$,
	\begin{equation}\label{eq:KPIC02}
		K_{\min}> \frac{n(n-1)}{n^{2}-n+12} S_{0},
	\end{equation}
	then the normalized Ricci flow starting at $g$ exists for all time and converges smoothly to a metric of positive constant sectional curvature. Therefore, $M$ is diffeomorphic to a spherical space form. Moreover is $M$ is simply connected, then it is diffeomorphic to $\mathbb{S}^{n}$.
\end{corollary}

No assumption on the $\pi_{1}(M)$ is needed in \autoref{cor:RicciFlow}. Li passes from the same scalar condition to a more general, unweighted four-frame hypothesis and then to positive isotropic curvature, obtaining a homeomorphism conclusion \cite[Thm1.4, Prop3.2]{Li2026}. Our proof does not use that implication. Instead, after writing $\mathrm{Rm}=K_{\min}I+E$, it retains the termwise inequality $E_{ijij}\ge 0$ and use a $(\lambda, \mu)$-dependent anisotropic rescaling of the homogeneous Ni-Wilking's estimate. This directly controls the full PIC2 family.

The paper is organized as follows. Section 2 fixes the curvature convention and introduces the model tensors. Section 3 establishes the homogeneous four-frame estimate, the sharp PIC2 criterion. Section 4 proves the sharp $q_2$ estimate together with its equality case. Section 5 applies these algebraic results to prove the global vanishing and rigidity theorems.

\section{Curvature Conventions and Model Tensors}

\begin{definition}\label{def:AlgebraicCurvatureTensor}
	Let $(V, \left\langle , \right\rangle)$ be a finite-dimensional Euclidean vector space. A multilinear map 
	\[
		E: V \times V \times V \times V \to \mathbb{R}
	\]
	is called an \emph{algebraic curvature tensor} if, for all vectors $X, Y, Z, W \in V$, it satisfies
	\begin{enumerate}
		\item $ E(X, Y, Z, W) = -E(Y, X, Z, W) $;
		\item $ E(X, Y, Z, W) = -E(X, Y, W, Z) $;
		\item $E(X, Y, Z, W) = E(Z, W, X, Y)$;
		\item $ E(X, Y, Z, W) + E(Y, Z, X, W) + E(Z, X, Y, W) = 0.$
	\end{enumerate}
\end{definition}

In components $E_{ijkl} = E(e_i, e_j, e_k, e_l)$, these identities read 
\begin{equation}\label{eq:AlgCur}
	\begin{aligned}
	E_{ijkl} = -E_{jikl}\qquad E_{ijkl} &= -E_{ijlk}\qquad E_{ijkl} = E_{klij}\\
	E_{ijkl} + E_{jkil} + E_{kijl} &= 0
	\end{aligned}
\end{equation}

For a Riemannian manifold $(M, g)$ with Levi-Civita connection $\nabla$, we distinguish the curvature endomorphism of the connection from its covariant form. Our convention is
\begin{equation*}
	R^{\nabla}(X, Y)Z = \nabla_{X}\nabla_{Y}Z-\nabla_{Y}\nabla_{X}Z-\nabla_{[X, Y]}Z,
\end{equation*}
and the Riemannian curvature tensor is the $(0, 4)$-tensor
\begin{equation*}
	\operatorname{Rm}(X, Y, Z, W) = \left\langle R^{\nabla}(X, Y)W, Z \right\rangle.
\end{equation*}

Therefore in an orthonormal basis $\{e_{i}\}$, components of the curvature tensor are given by:
\begin{equation*}
	\mathrm{Rm}_{ijkl}= \left\langle R^{\nabla}(e_{i}, e_{j})e_{l}, e_{k} \right\rangle,
\end{equation*}
and we define its contraction fiberwise by 
\begin{equation*}
	\mathrm{Ric}(X, Y) = \sum_{i} \mathrm{Rm}(X, e_{i}, Y, e_{i})
\end{equation*}
So $K(\mathrm{span}(e_{i}, e_{j}))= \mathrm{Rm}_{ijij}$ and the scalar curvature is given by:
\begin{equation*}
	\sca = 2\sum_{i<j} \mathrm{Rm}_{ijij}.
\end{equation*}

For an algebraic curvature tensor $E$ and linear independent $X, Y\in V$, define 
\[
	K_{E}(\mathrm{span}(X, Y)) = \frac{E(X,Y,X,Y)}{|X|^{2}|Y|^{2}-\left\langle X, Y \right\rangle^{2}}, ,
\] 
Its Ricci and scalar contractions are $(0,2)$-tensor and scalar, respectively,
\begin{equation*}
	\mathrm{Ric}(E)(X, Y) = \sum_{i}E(X, e_{i}, Y, e_{i}),\qquad \sca(E)=2\sum_{i<j} E_{ijij}
\end{equation*}
where the scalar curvature formula uses any orthonormal basis. We use the metric to raise and lower indices. In particular $\mathrm{Ric}(E)^{\sharp}\in \mathrm{End}(V)$is defined by 
\begin{equation*}
	\left\langle \mathrm{Ric}(E)^{\sharp}X, Y \right\rangle = \mathrm{Ric}(E)(X, Y).
\end{equation*}

Fix an orthonormal basis $\{e_{i}\}$ of $V$, let $\{e^{i}\}$ be its dual and write $e^{ij}:=e^{i}\wedge e^{j}$. The self-adjoint \emph{curvature operator}
\begin{equation*}
	\mathcal{R}_{E}: \Lambda^{2}V^{*} \to \Lambda^{2}V^{*}
\end{equation*}
is defined by 
\begin{equation*}
	\left\langle \mathcal{R}_{E}(X^{\flat}\wedge Y^{\flat}), Z^{\flat}\wedge W^{\flat} \right\rangle = E(X, Y, Z, W).
\end{equation*}
Therefore $\left\langle \mathcal{R}_{E}(e^{ij}), e^{kl} \right\rangle = E_{ijkl}$ for $i<j$ and $k<l$. Equivalently, if 
\begin{equation*}
	\omega = \frac{1}{2} \sum_{k,l}\omega_{kl}e^{k}\wedge e^{l},
\end{equation*}
then
\begin{equation*}
	(\mathcal{R}_{E}\omega)_{ij} = \frac{1}{2}\sum_{k,l}E_{ijkl}\omega_{kl}.
\end{equation*}

For $B\in \mathrm{End}(V) $, denotes its induced action on two-forms by 
\begin{equation*}
	(B^{[2]}\omega)(X, Y) = \omega(BX, Y) + \omega(X, BY).
\end{equation*}

\begin{definition}\label{def:WeitCurv}
	For an algebraic curvature tensor $E$, define its \emph{Weitzenb\"{o}ck curvature endomorphism} acting on two-forms by 
	\begin{equation}\label{eq:WeitCurv1}
		q_{2}(E):= \left( \mathrm{Ric}(E)^{\sharp} \right)^{[2]}- 2 \mathcal{R}_{E}.
	\end{equation}
\end{definition}
	Indeed, the first Bianchi identity implies that
	\begin{equation*}
		\sum_{kl}E_{ikjl}\omega_{kl}=\sum_{kl}\frac{1}{2}E_{ijkl}\omega_{kl} = \left( \mathcal{R}_{E}\omega \right)_{ij}
	\end{equation*}
	and therefore
	\begin{equation}\label{eq:Bochner}
		(q_{2}(E)\omega)_{ij}= \operatorname{Ric}_{i}^{k}\omega_{kj} + \operatorname{Ric}_{j}^{k}\omega_{ik} - 2E_{ikjl}\omega_{kl}.
	\end{equation}

For a Riemannian manifold $(M, g)$, we take $\nabla^{*}\nabla$ to be the nonnegative rough Laplacian. For a two-form $\omega\in \Omega^{2}(M)$, the Hodge Laplacian satisfies
\begin{equation}\label{eq:WZBKFormula}
	\Delta_{\mathrm{H}}\omega = \nabla^*\nabla \omega + q_{2}(\mathrm{Rm})\omega,
\end{equation}
where $q_{2}(\mathrm{Rm})$ is defined by \eqref{eq:WeitCurv1}. For background on this formula and its algebraic term, see \cite{Boc1946}, \cite{Pet2016}, \cite{NPW2023}. Both $E\to \mathcal{R}_{E}$ and $E\to q_{2}(E)$ are linear.

If $I$ denotes the constant sectional curvature one tensor, then
\begin{equation}\label{eq:CurvTensI}
	I(X, Y, Z, W)= \left\langle X, Z \right\rangle \left\langle Y,W \right\rangle-\left\langle X,W \right\rangle\left\langle Y,Z \right\rangle.
\end{equation}
In components $I_{ijkl}=\delta_{ik}\delta_{jl}-\delta_{il}\delta_{jk}$. By direct contraction, for $m\in \mathbb{R}$, we have 
\[
	q_{2}(m I) = 2(n-2) m \mathrm{Id}|_{\Lambda^{2}V^{*}}.
\]

\begin{definition}\label{def:J}
	Let $V$ be an Euclidean vector space of dimension $n$. For a skew-adjoint endomorphism $J\in \mathrm{End}(V)$, define the algebraic curvature tensor $E_{J}$ by 
	\begin{equation}\label{eq:EJ}
		\begin{aligned}
			E_{J}(X,Y,Z,W) &= \left\langle JX, Z \right\rangle\left\langle JY,W \right\rangle
			- \left\langle JX, W \right\rangle\left\langle JY,Z \right\rangle\\
				       &+2\left\langle JX, Y \right\rangle\left\langle JZ, W \right\rangle
		\end{aligned}
	\end{equation}
\end{definition}
Direct calculation verifies $E_{J}$ is indeed an algebraic curvature tensor. The following properties of $E_{J}$ can be verified easily: If $J$ is an orthogonal complex structure, then for orthonormal $X, Y$,
\begin{equation}\label{eq:JPro01}
	K_{E_{J}}(X, Y) = 3\left\langle JX, Y \right\rangle^{2},\quad \operatorname{Ric}(E_{J}) = 3g,\quad \sca(E_{J}) = 3n.
\end{equation}
More generally, suppose that $W\subset V$ is a $J$-invariant subspace and that $J^{2}=-P_{W}$, where $P_{W}$ is the orthogonal projection onto $W$. Then
\begin{equation}\label{eq:JPro02}
	\operatorname{Ric}(E_{J})(X, Y) = 3 \left\langle P_{W}X, Y \right\rangle,\quad \sca(E_{J}) = 3 \dim W.
\end{equation}
These identities will be used in the discussion of equality cases.

\section{Four-Frame Estimates and the Sharp PIC2 Criterion}
In this section, our main goal is to prove a sharp estimate for the Weitzenb\"{o}ck curvature term $q_{2}(E)$. 

Since the pinching condition is given in terms of sectional curvatures, we need to control the off-diagonal component of $E$. First, we work on a fixed four-dimensional subspace of $\mathbb{R}^{n}$. The key input is the generalized Berger inequality of Ni-Wilking \cite{NW2010}, Corollary 2.2. Their statement is homogeneous and applies to four mutually orthogonal vectors; it imposes no unit length normalization.

\begin{lemma}\label{lem:NW}
	Let $E$ be an algebraic curvature tensor with nonnegative sectional curvature. Put 
	\begin{equation*}
		k=E_{1212}, \qquad l=E_{3434}, \qquad r=E_{1234},
	\end{equation*}
	and
	\begin{equation*}
		A=E_{1313}+E_{1414}+E_{2323}+E_{2424}.
	\end{equation*}
	Then
	\begin{equation}\label{eq:NW00}
		6|r|\le A+ 4\sqrt{kl} .
	\end{equation}
\end{lemma}
\begin{proof}
	For an algebraic curvature tensor $F$ and a unit vector $X$, the flag curvature form of $F$ is the symmetric form $F_{X}(Y, Z)=F(X,Y,X,Z)$ on $X^{\perp}$. We say that $F$ has $\lambda$-pinched flag curvature if the least and greatest eigenvalues of every such form satisfy $\mu_{\min}\ge \lambda \mu_{\max}$. The pointwise estimate in \cite{NW2010} states that if 
$F$ has nonnegative sectional curvature and $\lambda$-pinched flag curvature, then any four mutually orthogonal vectors $X, Y, Z, W$ satisfy
	\begin{equation}\label{eq:NW01}
	\begin{aligned}
		6\frac{1+\lambda}{1-\lambda}|F(X,Y,Z,W)| \leq{} & K_F(X,Z)+K_F(Y,Z)+K_F(X,W)\\
								&+K_F(Y,W)+2K_F(X,Y)+2K_F(Z,W),
	\end{aligned}
	\end{equation}
	where $K_{F}(U, V)=F(U, V, U, V)$. Their proof is pointwise and algebraic, so the estimate applies here. For $\varepsilon>0$ small, we set 
	\begin{equation}\label{eq:NW02}
		K_{\max}(E)=\max_{\sigma\in \mathrm{Gr}_{2}(\mathbb{R}^{n})} K_{E}(\sigma), \qquad E_{\varepsilon}=E+ \varepsilon I.
	\end{equation}
	The eigenvalues of every flag curvature form of $E_{\varepsilon}$ lie in the interval 
	\[
		[\varepsilon , K_{\max}(E)+\varepsilon].
	\]
	Consequently, $E_{\varepsilon}$ has positive $\lambda_{\varepsilon}$-pinched flag curvature, for example with
	\begin{equation*}
		\lambda_{\varepsilon} = \frac{\varepsilon}{K_{\max}(E)+2\varepsilon}\in (0, 1).
	\end{equation*}
	For $t>0$, apply \eqref{eq:NW01} to 
	\begin{equation*}
		X=t^{1/4}e_{1}, \qquad Y= t^{1/4}e_{2},\qquad Z= t^{-1/4}e_{3},\qquad W= t^{-1/4}e_{4},
	\end{equation*}
	Since $I(e_{1}, e_{2}, e_{3}, e_{4})=0$ and $(1+\lambda_{\varepsilon})/(1-\lambda_{\varepsilon})\ge 1$, this gives 
	\begin{equation*}
		6|r| \le A+4\varepsilon +2t (k+\varepsilon) + 2 \frac{l+\varepsilon}{t}.
	\end{equation*}
	Letting $\varepsilon\to 0$ yields:
	\begin{equation*}
		6|r| \le A+2tk+2 \frac{l}{t}, \qquad (t>0).
	\end{equation*}
	Taking the infimum and using $\inf_{t>0}(2tk+\frac{2l}{t})=4\sqrt{kl}$ proves \eqref{eq:NW00}.
\end{proof}

For an algebraic curvature tensor $E$ and an orthonormal four-frame $e_{1}, e_{2}, e_{3}, e_{4}$, define:
\begin{equation}\label{eq:thmPIC01}
	\mathcal{I}_{\lambda, \mu}(E) = E_{1313} + \lambda^{2}E_{1414} + \mu^{2}E_{2323} + \lambda^{2}\mu^{2}E_{2424} -2\lambda\mu E_{1234}.
\end{equation}
The tensor $E$ lies in PIC2 cone if $\mathcal{I}_{\lambda, \mu}(E) \ge 0$ for every orthonormal four-frame and every $\lambda, \mu\in [-1, 1]$; it lies in the interior of PIC2 if all these inequalities are strict.

\begin{proof}[Proof of \autoref{thm:PIC2}]
	For simplicity we set $E=R-K_{\min}I$, $\tau = \sca(E)$. Then the tensor $E$ has nonnegative sectional curvature, and 
	\begin{equation*}
		\sca(E) = n(n-1)(S_{0}(R) - K_{\min}).
	\end{equation*}
	Therefore, the pinching condition \eqref{eq:KPIC01} is equivalent to 
	\begin{equation}\label{eq:thmPIC02}
		\tau \le 12K_{\min}.
	\end{equation}
	Since $\tau\ge 0$, this gives $K_{\min}\ge 0$, under the strict pinching one has $\tau < 12 K_{\min}$ and hence $K_{\min} >0$. Changing the signs of $\lambda$ and $\mu$ can be absorbed by reversing one vector of the four-frame: the diagonal terms depend only on the squares of the parameters, while such a reversal changes the sign of $E_{1234}$. Hence it suffices to treat  $0\le \lambda, \mu \le 1$. The diagonal term in \eqref{eq:thmPIC01} depend only on the square of the parameters; if $\lambda\mu < 0$, reversing one vector of the four-frame changes the sign of $R_{1234}$ without changing any diagonal term. Fix such a four-frame, extend it to an orthonormal basis. Put $k=E_{1212}$, $l=E_{3434}$, $r=E_{1234}$, and
	\begin{equation*}
		D_{\lambda, \mu}:= E_{1313}+\lambda^{2}E_{1414} + \mu^{2} E_{2323} +\lambda^{2}\mu^{2}E_{2424}.
	\end{equation*}
	Nonnegative sectional curvature implies that
	\begin{equation}\label{eq:thmPIC03}
		\tau = 2\sum_{i<j}E_{ijij} \ge 2(k+l).
	\end{equation}
	Suppose first that $\lambda \mu >0$. For $E_{\varepsilon}$ and $\lambda_{\varepsilon}$ as in the proof of \autoref{lem:NW} and apply \eqref{eq:NW01}, with $F=E_{\varepsilon}$, to 
	\begin{equation*}
		X=t^{1/4}e_{1},\qquad Y=t^{1/4}\mu e_{2}, \qquad Z=t^{-1/4}e_{3},\qquad W= t^{-1/4}\lambda e_{4}
	\end{equation*}
	where $t>0$. We have
	\begin{equation}\label{eq:thmPIC07}
		\begin{aligned}
		6 \frac{1+\lambda_{\varepsilon}}{1-\lambda_{\varepsilon}}\lambda\mu |r| &\le D_{\lambda, \mu} + \varepsilon(1+\lambda^2)(1+\mu^2)\\
		&+2t\mu^2	(k+\varepsilon) + 2 \frac{\lambda^2}{t}(l+\varepsilon)
		\end{aligned}
	\end{equation}
	Dropping $(1+\lambda_{\varepsilon})/(1-\lambda_{\varepsilon})\ge 1$ and letting $\varepsilon\to 0$ gives
	\begin{equation*}
		6\lambda \mu |r| \le D_{\lambda, \mu} + 2t \mu^{2} k + 2 \frac{\lambda^{2}}{t} l, \qquad t>0.
	\end{equation*}
	Chooseing $t=\lambda/\mu$ balances the two parameter weights and gives:
	\begin{equation}\label{eq:thmPIC04}
		6\lambda \mu |r| \le D_{\lambda, \mu} + 2\lambda \mu (k +l).
	\end{equation}
	Therefore by \eqref{eq:thmPIC03},
	\begin{equation}\label{eq:thmPIC05}
		\begin{aligned}
			\mathcal{I}_{\lambda, \mu}(E) &= D_{\lambda, \mu} - 2\lambda \mu r\\
						      &\ge \frac{2}{3}D_{\lambda, \mu} - \frac{2}{3}\lambda \mu (k+l)\\
						      &\ge \frac{2}{3}D_{\lambda, \mu} - \frac{\lambda\mu}{3}\tau \\
						      &\ge -\frac{\lambda\mu}{3}\tau 
		\end{aligned}
	\end{equation}
	If $\lambda\mu =0$, then $\mathcal{I}_{\lambda, \mu}(E) = D_{\lambda,\mu} \ge 0$, so \eqref{eq:thmPIC05} remains valid. Since $\mathcal{I}_{\lambda, \mu}(K_{\min}I) = K_{\min}(1+\lambda^{2})(1+\mu^{2})$, we obtain
	\begin{equation}\label{eq:thmPIC06}
		\begin{aligned}
		\mathcal{I}_{\lambda, \mu} (R) & \ge K_{\min}(1+\lambda^{2})(1+\mu^{2}) - \frac{\lambda \mu}{3}\tau\\
					       &= K_{\min}\left( (1-\lambda\mu)^{2} + (\lambda-\mu)^{2} \right) + \frac{\lambda \mu}{3}(12K_{\min} - \tau) \ge 0.
		\end{aligned}
	\end{equation}
	This proves weak PIC2 assertion. Under strict pinching, the last quantity is positive when $\lambda\mu >0$ since $12K_{\min} > \tau$. When $\lambda \mu =0$, it is still positive since $K_{\min}>0$. Thus the strict pinching implies the interior of PIC2. 
	
	It remains to show sharpness. Fix a four dimensional subspace $W\subset V$ with orthonormal basis $e_{1}, e_{2}, e_{3}, e_{4}$ and define \[Je_{1}=e_{2}, Je_{2}=-e_{1}, Je_{3}=e_{4}, Je_{4}=-e_{3},\quad J|_{W^{\perp}}=0.\]
	Using \autoref{def:J}, and for $t>0$ set	
	\begin{equation}\label{eq:ABS}
		R_{t}= I+tE_{J}.
	\end{equation}
	For orthonormal $X, Y$, one has
	\begin{equation*}
		K_{R_{t}} (\mathrm{span}(X, Y)) = 1+ 3t\left\langle JX, Y \right\rangle^{2}.
	\end{equation*}
	Therefore $K_{\min}(R_{t})=1$. A direct calculation shows
	\begin{equation*}
		\sca(R_{t}) = n(n-1) + 12 t,\qquad S_{0}(R_{t}) = 1 + \frac{12t}{n(n-1)}
	\end{equation*}	
	On the given four-frame we calculate
	\begin{equation*}
		\mathcal{I}_{1,1}(R_{t}) = 4(1-t).
	\end{equation*}
	For $t=1$,
	\begin{equation*}
		\frac{K_{\min}(R_{1})}{\sca_{0}(R_1)}=\frac{n(n-1)}{n(n-1)+12},\qquad \mathcal{I}_{1,1}(R_{1}) = 0,
	\end{equation*}
	so the endpoint can lie on the boundary of PIC2. Given any $0<\gamma < n(n-1)/(n(n-1)+12)$, choose $t>1$ sufficiently close to $1$ that $n(n-1)/(n(n-1)+12t) >\gamma$.
	Then
	\begin{equation*}
		K_{\min}(R_{t}) > \gamma S_{0}(R_{t}),\qquad \mathcal{I}_{1,1}(R_{t})<0.
	\end{equation*}
	Therefore $R_{t}$ does not even have nonnegative isotropic curvature, proving the claimed optimality in every dimension $n\ge 4$.
\end{proof}

\begin{proof}[Proof of \autoref{cor:RicciFlow}]
	Apply \autoref{thm:PIC2} to $\mathrm{Rm}_x$ at every point $x\in M$. The strict hypothesis places the curvature tensor in the interior of PIC2, so the convergence theorem of Brendle-Schoen \cite{BS2009a} applies here: the normalized Ricci flow exists for all time and converges smoothly to a metric of constant sectional curvature. Note that taking $\lambda=\mu=0$ in \eqref{eq:thmPIC01} shows that the initial metric has positive sectional curvature. Hence the fundamental group is finite, hence its universal cover if compact. The remaining conclusions follow from the classification of complete positive constant-curvature manifolds.
\end{proof}

\section{The Sharp Two-Form Weitzenb\"{o}ck Estimate}
Set $\ell = \lfloor\frac{n}{2}\rfloor$, so $n=2\ell$ or $n=2\ell +1$.

\begin{theorem}[Sharp $q_{2}$ estimate]\label{thm:SharpEstimate}
	Let $E$ be an algebraic curvature tensor on $\mathbb{R}^{n}$, $n\ge 4$, with nonnegative sectional curvature. As a quadratic form inequality on $\Lambda^{2}(\mathbb{R}^{n})^{*}$, 
	\begin{equation}\label{eq:SharpEstimate01}
		q_{2}(E) \ge -\frac{2(\ell -1)}{3\ell} \sca(E) \mathrm{Id}.
	\end{equation}
	The coefficient is sharp in every dimension.
\end{theorem}
\begin{proof}
	Let $\omega \in \Lambda^{2}(\mathbb{R}^{n})^{*}$ be a unit two-form. Choose an orthonormal basis 
	\begin{equation*}
		\left\{ e_{1}, \dots, e_{n} \right\}\qquad \text{with dual basis}\qquad \left\{ e^{1}, \dots, e^{n} \right\},
	\end{equation*}
	such that $\omega$ can be written as 
	\begin{equation}\label{eq:SharpEstimate02}
		\omega = \sum_{a=1}^{\ell} \lambda_{a}e^{2a-1}\wedge e^{2a},\quad \sum_{a=1}^{\ell} \lambda_{a}^{2}=1.
	\end{equation}
	If $n$ is odd, $e_{n}$ is unused. For $1\le a \le \ell$, let $\sigma_{a} = \mathrm{span}\{e_{2a-1} e_{2a}\}$ and set
	\begin{equation*}
		k_{a} := E_{2a-1, 2a, 2a-1, 2a}.
	\end{equation*}
	For $a<b$, let $A_{ab}$ be the sum of the four coordinate sectional curvatures having direction from $\sigma_{a}$ and one form $\sigma_{b}$, and put the off-diagonal term:
	\begin{equation*}
		r_{ab} := E_{2a-1, 2a, 2b-1, 2b}.
	\end{equation*}
	Expandding \eqref{eq:Bochner} and writing the Ricci terms as sums of sectional curvatures, gives the following:
	\begin{equation}\label{eq:SharpEstimate2.5}
		\begin{aligned}
			\left\langle q_{2}(E)\omega, \omega \right\rangle 
			= &\sum_{a<b}\left[ (\lambda_{a}^{2}+\lambda_{b}^{2})A_{ab} - 4\lambda_{a}\lambda_{b}r_{ab} \right]\\
			  &+ \frac{1-(-1)^{n}}{2}\sum_{a} \lambda_{a}^{2} (E_{2a-1, n, 2a-1, n} + E_{2a, n, 2a, n}).
		\end{aligned}
	\end{equation}
	The last line is nonnegative. Applying \autoref{lem:NW} to each pair of blocks gives:
	\begin{equation}\label{eq:SharpEstimate03}
	\begin{aligned}
	&(\lambda_a^2+\lambda_b^2)A_{ab}-4\lambda_a\lambda_b r_{ab}\\
	&\quad\geq \left(\lambda_{a}^2+ \lambda_{b}^2-\frac{2}{3} |\lambda_{a}\lambda_{b}|\right)A_{ab} -\frac{8}{3} |\lambda_{a}\lambda_{b}|\sqrt{k_ak_b}\\
	&\quad\geq-\frac{8}{3} |\lambda_{a}\lambda_{b}|\sqrt{k_ak_b}.
	\end{aligned}
	\end{equation}
	Apply Cauchy-Schwarz we have
	\begin{equation}\label{eq:SharpEstimate04}
	\begin{aligned}
		\sum_{a<b} |\lambda_{a}|\sqrt{k_{a}} |\lambda_{b}|\sqrt{k_{b}} 
		&\le \frac{\ell -1}{2\ell} \left( \sum_{a} |\lambda_{a}|\sqrt{k_{a}} \right)^{2}\\
		&\le \frac{\ell -1}{2\ell}\left( \sum_{a}\lambda_{a}^{2} \right)\left( \sum_{a}k_{a} \right)\\
		&\le \frac{\ell -1}{4\ell}\sca(E).
	\end{aligned}
	\end{equation}
	Here we have used $\sum_{a} \lambda_{a}^{2}=1$ and $\sum_{a} k_{a} \le \sca/2$. Plug it back to \eqref{eq:SharpEstimate03}, we get desired inequality. 

	It remains to show the estimate is sharp. The model is the curvature tensor of $\mathbb{CP}^{\ell}$ (after removing the identity curvature contribution). Let $W\subset \mathbb{R}^{n}$ be a $2\ell$-dimensional subspace, choose an orthogonal complex structure $J_{W}\in \mathrm{End}(W)$, and set $J=J_{W}\oplus 0_{W^{\perp}}\in \mathrm{End}(\mathbb{R}^{n}).$ If $P_{W}$ is the orthogonal projection onto $W$, then 
	\begin{equation*}
		J^{*}=-J,\qquad J^{2}=-P_{W}, 
	\end{equation*}
	where $J^{*}$ denotes the adjoint of the endomorphism $J$. Note that $J$ is a complex structure on $\mathbb{R}^{n}$ if and only if $n=2\ell$ is even. When $n$ is odd, it is a rank $2\ell$ skew-adjoint endomorphism. Choose an adapted orthonormal basis such that
	\begin{equation*}
		Je_{2a-1}=e_{2a},\quad Je_{2a}=-e_{2a-1}\qquad (1\le a\le \ell),
	\end{equation*}
	and $Je_{n}=0$ when $n$ is odd. Use the algebraic curvature tensor defined in \autoref{def:J}. For orthonormal $X, Y$
	\begin{equation*}
		K_{E_{J}}(X, Y) = 3 \left\langle JX, Y \right\rangle^{2} \ge 0.
	\end{equation*}
	Define the associated two-form by $\Omega_{J}(X, Y)=\left\langle JX, Y \right\rangle$. In the adapted coframe,
	\begin{equation*}
		\Omega_{J}=\sum_{a=1}^{\ell} e^{2a-1}\wedge e^{2a}, \qquad \omega_{J}=\ell^{-1/2}\Omega_{J}.
	\end{equation*}
	Here $|\Omega_{J}|^{2}=\ell$, so $\omega_{J}$ is of unit norm. For $a<b$ the relevant curvature quantities are
	\[
		k_{a}=k_{b}=3,\qquad A_{ab}=0, \qquad  r_{ab}=2.
	\]
	Moreover by \eqref{eq:JPro01},
	\begin{equation*}
		\mathrm{Ric}(E_{J})(X, Y)=3\left\langle P_{W}X, P_{W}Y \right\rangle, \qquad  \sca(E_{J})=6\ell,
	\end{equation*}
	Therefore the scalar curvature is $3n$ when $n$ is even and $3(n-1)$ when $n$ is odd. The unused direction contributes zero in the odd case. Substituting these values into \eqref{eq:SharpEstimate2.5}, we obtain
	\begin{equation*}
		\left\langle q_{2}(E_{J})\omega_{J}, \omega_{J} \right\rangle  = - 4 (\ell -1).
	\end{equation*}
	This equals $-\frac{2(\ell-1)}{3\ell}\sca(E_{J})$, so the bound is sharp in every dimension.
\end{proof}

For the endpoint case, we need the following two lemmas 
\begin{lemma}\label{lem:Null}
	Let $E$ be an algebraic curvature tensor with nonnegative sectional curvature. 
	\begin{enumerate}
	\item If $E_{1212}=0$, then $E_{121j}=E_{122j}=0$, for $j\ge 3$.
	\item Suppose $\dim V=4$ and $e_{1}, e_{2}, e_{3}, e_{4}$ is an orientated orthonormal basis. Let $J$ be the orthogonal complex structure defined by $Je_{1}=e_{2}, Je_{3}=e_{4}$. If
		\[
			E_{1212}=E_{3434}=k>0,\quad E_{1313}=E_{1414}=E_{2323}=E_{2424}=0,
		\] 
		and the off-diagonal term $E_{1234}=\frac{2k}{3}$. Then
		\[
			E=\frac{k}{3}E_{J},
		\]
		where $E_{J}$ is the algebraic curvature tensor defined in \eqref{eq:EJ}.
	\end{enumerate}
\end{lemma}
\begin{proof}
	(1) For a fixed $j\ge 3$ and every $t\in \mathbb{R}$, the vectors $e_{1}$ and $e_{2}+te_{j}$ are linearly independent, so 
	\begin{equation*}
		E(e_{1}, e_{2}+te_{j},e_{1}, e_{2}+te_{j})\ge 0.
	\end{equation*}
	Expanding and using the curvature symmetries gives:
	\begin{equation*}
		\begin{aligned}
			E(e_{1}, e_{2}+te_{j},e_{1}, e_{2}+te_{j}) &= E_{1212} +2t E_{121j} + t^{2}E_{1j1j}\\
								   &= 2t E_{121j} + t^{2}E_{1j1j}.
		\end{aligned}
	\end{equation*}
	This quadratic is nonnegative for every $t$, hence $E_{121j}=0$. The equality $E_{122j}=0$ can be proved similarly.

	(2) By part (1), the four vanishing coordinate sectional curvatures, together with the pair and skew symmetries, put every component having exactly three distinct indices into one of the vanishing forms covered by (1). Thus all such components vanish. We can then compute the Ricci tensor; for instance $\mathrm{Ric}_{11} = E_{1212} + E_{1313} + E_{1414} = k$. By symmetry, $\mathrm{Ric}_{i i}=k$ for $i=1,2,3,4$, while for $i\ne j$
	\begin{equation*}
		\mathrm{Ric}_{ij} =\sum_{m=1}^{4}E_{imjm}=0.
	\end{equation*}
	Therefore 
	\begin{equation}\label{eq:RicOp}
		\operatorname{Ric}(E) = k \left\langle , \right\rangle, \qquad \mathrm{Ric}(E)^{\sharp} = k \mathrm{Id},
	\end{equation}
	on this four-dimensional space. Let $*: \Lambda^{2}V^{*}\to \Lambda^{2}V^{*}$ be the Hodge star determined by the given orientation. The standard four-dimensional decomposition of the curvature operator, cf. Chapter 1 of \cite{Bes1987}, identifies its off-diagonal blocks with the trace-free Ricci tensor. Hence \eqref{eq:RicOp} implies that $\mathcal{R}_{E}$ preserves 
	\[
		\Lambda^2 = \Lambda^{+}\oplus \Lambda^{-},\qquad \Lambda^{\pm}=\left\{ \xi\in \Lambda^{2}V^{*}\ : \ *\xi = \pm \xi \right\}.
	\]
	Write $e^{ij}=e^{i}\wedge e^{j}$ and choose the orthonormal bases
	\[
	\begin{aligned}
		\omega_1&=\frac{1}{\sqrt{2}}(e^{12}+e^{34}),&
		\omega_2&=\frac{1}{\sqrt{2}}(e^{13}-e^{24}),&
		\omega_3&=\frac{1}{\sqrt{2}}(e^{14}+e^{23})
	\end{aligned}
	\]
	of $\Lambda^{+}$ and
	\[
	\begin{aligned}
		\eta_1&=\frac{1}{\sqrt{2}}(e^{12}-e^{34}),&
		\eta_2&=\frac{1}{\sqrt{2}}(e^{13}+e^{24}),&
		\eta_3&=\frac{1}{\sqrt{2}}(e^{14}-e^{23})
	\end{aligned}
	\]
	of $\Lambda^{-}$. Set $r=E_{1234}=2k/3$, $s=E_{1324}$, and $u= E_{1423}$. The first Bianchi identity gives $r-s+u=0$, which implies $u = s - 2k/3$. Define the curvature operator blocks
	\begin{equation*}
		\mathcal{R}_{E}^{+} = \mathcal{R}_{E}|_{\Lambda^{+}}, \qquad \mathcal{R}_{E}^{-} = \mathcal{R}_{E}|_{\Lambda^{-}}.
	\end{equation*}
	Because all curvature components with exactly three distinct indices are zero, a direct calculation shows that they can be written in the matrix form under the basis $\{\omega_{i}\}$ and $\{\eta_{i}\}$ respectively
	\[
		\mathcal{R}_{E}^{+}=\begin{pmatrix}
		k+r&0&0\\
		0&-s&0\\
		0&0&u
		\end{pmatrix},
		\qquad
		\mathcal{R}_{E}^{-}=
		\begin{pmatrix}
		k-r&0&0\\
		0&s&0\\
		0&0&-u
		\end{pmatrix}.
	\]
	We can directly compute the spectra of these two diagonal curvature blocks:
	\begin{equation*}
		\text{Spec}(\mathcal{R}_{E}^{+}) = \left( \frac{5k}{3}, -s, s-\frac{2k}{3} \right),\qquad \text{Spec}(\mathcal{R}_{E}^{-}) = \left( \frac{k}{3}, s, \frac{2k}{3}-s \right).
	\end{equation*}
	Let $\mu_{\pm}$ denotes $\lambda_{\min}(\mathcal{R}_{E}^{\pm})$. We claim that the nonnegativity of the sectional curvature is equivalent to
	\begin{equation}\label{eq:min-eigenvalue-condition}
		\mu_{+}+ \mu_{-}\geq 0.
	\end{equation}
	Geometrically, every unit simple two-form $\xi\in\Lambda^{2}V^{*}$ can be written as
	\[
		\xi=\frac{1}{\sqrt{2}}(\alpha+\beta),\qquad \alpha\in\Lambda^{+},\quad \beta\in\Lambda^{-},\quad |\alpha|=|\beta|=1.
	\]
	Conversely, every expression of this form is a unit decomposable two-form. Under the metric indentification it represents an oriented two-plane, and 
	\begin{equation*}
		\left\langle \mathcal{R}_{E}\xi, \xi\right\rangle = \frac{1}{2}\left( \left\langle \mathcal{R}_{E}^{+}\alpha, \alpha \right\rangle 
		+ \left\langle \mathcal{R}_{E}^{-}\beta, \beta \right\rangle\right).
	\end{equation*}
	Minimizing over all unit $\alpha$ and $\beta$ gives
	\[
		\min_{\sigma\in \mathrm{Gr}_{2}(\mathbb{R}^{4})}K_E(\sigma)=\frac{1}{2}(\mu_{+}+\mu_{-}).
	\]
	This proves \eqref{eq:min-eigenvalue-condition}. In the self-dual block, the sum of the last two eigenvalues is $-2k/3$, so $\mu_{+}\le -k/3$. Similarly, the sum of the last two eigenvalues in the anti-self-dual block is $2k/3$, which implies $\mu_{-} \le k/3$. Consequently, $\mu_{+}+\mu_{-} \le 0$, whereas nonnegative sectional curvature forces the opposite inequality. Equality must therefore hold in both preceding eigenvalue bounds. The last two eigenvalues in each block are equal, so $-s = s - 2k/3$, whence $s = k/3$ and $u=-k/3$. These components determine the tensor in dimension four and agree with those of $kE_{J}/3$. Hence $E=kE_{J}/3$.
\end{proof}

\begin{remark}
	The sharpness is pointwise and algebraic. It does NOT by itself imply that the global pinching threshold $\beta_{n}$ is topologically optimal.
\end{remark}

\begin{lemma}\label{lem:Even}
	Let $V$ be a $2\ell$-dimensional Euclidean space. Suppose that an algebraic curvature tensor $E$ on $V$ has nonnegative sectional curvature and positive scalar curvature, and let $\omega\in \Lambda^{2}V^{*}$ be a unit, full rank two-form. Here full rank means that skew-adjoint endomorphism $B_{\omega}$ defined by 
	\begin{equation*}
		\omega(X, Y) = \left\langle B_{\omega}X, Y \right\rangle
	\end{equation*}
	is invertible. If $\omega$ satisfies 
	\begin{equation*}
		\left\langle q_{2}(E)\omega, \omega \right\rangle = - \frac{2(\ell-1)}{3\ell} \sca(E),
	\end{equation*}
	then, after orienting the two-plane blocks appropriately, there is an orthonormal basis $e_{1}, \cdots, e_{2\ell}$ such that
	\begin{equation*}
		\omega = \frac{1}{\sqrt{\ell}} \sum_{a=1}^{\ell}e^{2a-1}\wedge e^{2a}, \qquad E= c E_{J},\ c>0,
	\end{equation*}
	where $J= \sqrt{\ell}B_{\omega}$ is the orthogonal complex structure determined by $Je_{2a-1} = e_{2a}$.
\end{lemma}
\begin{proof}
	Choose an orthonormal basis in which $\omega$ has the normal form
	\[
	\omega =\sum_{a=1}^{\ell}\lambda_a\,e^{2a-1}\wedge e^{2a},\qquad \sum_{a=1}^{\ell}\lambda_a^2=1.
	\]
	Since $\omega$ has full rank, each $\lambda_a$ is nonzero. Set
	\[
		\sigma_a=\operatorname{span}\{e_{2a-1},e_{2a}\}.
	\]
	For later reference, define exactly as in the proof of \autoref{thm:SharpEstimate}
	\begin{equation*}
		k_{a}  = E_{2a-1, 2a, 2a-1, 2a},\quad A_{ab} = \sum_{p\in \{2a-1, 2a\}} \sum_{q\in \{2b-1, 2b\}} E_{pqpq},
	\end{equation*}
	\begin{equation*}
		r_{ab}= E_{2a-1,2a, 2b-1, 2b},\quad (a<b).
	\end{equation*}
	Trace the equality conditions in the proof of \autoref{thm:SharpEstimate}. First in \eqref{eq:SharpEstimate03}, the full rank assumption implies that $\lambda_{a}^{2}+\lambda_{b}^{2}-\frac{2}{3}|\lambda_{a}\lambda_{b}|>0$, equality in \eqref{eq:SharpEstimate03} forces $A_{ab}=0$ for every $a<b$. Second, in \eqref{eq:SharpEstimate04}, there are two inequalities, the first equality forces all $|\lambda_{a}|\sqrt{k_{a}}$ equal and the second equality makes $|\lambda_{a}|$ proportional to $\sqrt{k_{a}}$. Equality also holds in $\sum_{a}k_{a}\le \sca(E)/2$. Since $\omega$ has full rank and $\sca(E)>0$, no $k_{a}$ can vanish. Combining these conditions gives
	\begin{equation*}
		|\lambda_{a}|=\ell^{-\frac{1}{2}},\qquad k_{a}=k >0.
	\end{equation*}
	Equality in the pairwise Ni-Wilking bound gives $|r_{ab}|=2k/3$ and $\operatorname{sign}(r_{ab})=\operatorname{sign}(\lambda_{a}\lambda_{b})$. After reversing orientation of blocks if necessary, we may assume
	\[
		\lambda_{a}= \ell^{-1/2},\qquad r_{ab}=2k/3,\qquad a<b
	\]
	It remains to check the other components of $E$. Fix $a<b$ and let $V_{ab}=\sigma_{a}\oplus \sigma_{b}$. The four nonnegative sectional curvatures comprising $A_{ab}$ all vanish.  Thus on $V_{ab}$,
	\begin{equation*}
			E_{2a-1,2a, 2a-1, 2a}=E_{2b-1, 2b, 2b-1, 2b} = k,\quad
			E_{2a-1, 2a, 2b-1, 2b}=r_{ab}=\frac{2k}{3}.
	\end{equation*}
	Moreover $E_{ijij}=0$ whenever $e_{i}\in \sigma_{a}$ and $e_{j}\in \sigma_{b}$. By \autoref{lem:Null} (2), we get the equality of restricted curvature tensors:
	\begin{equation}\label{eq:evenRigi02}
		E|_{V_{ab}}=\frac{k}{3}E_{J}|_{V_{ab}}
	\end{equation}
	Part (1) of \autoref{lem:Null}, together with curvature symmetries also shows that
	\begin{equation}\label{eq:evenRigi03}
		E_{ijkl}=0\quad \text{whenever exactly three of}\ i,j,k,l \ \text{are distinct}.
	\end{equation}
	It remains to consider components involving four distinct basis vectors. Those with indices in two blocks have already been determined by \eqref{eq:evenRigi02}. Therefore, we only have to consider the two remaining possibilities.

	First suppose that the four indices lie in three blocks. Let $e_{p}, e_{q}$ be the two basis vectors in $\sigma_{a}$ and let $e_{r}\in \sigma_{b}$, $e_{s}\in \sigma_{c}$, where $a, b, c$ are distinct. Set
	\begin{equation*}
		x=E_{pqrs},\quad y=E_{prqs},\quad z=E_{psqr}.
	\end{equation*}
	The first Bianchi identity gives:
	\begin{equation*}
		x-y+z=0.
	\end{equation*}
	All coordinate sectional curvatures involving vectors from two different blocks vanish, and all components with three distinct indices vanish by \eqref{eq:evenRigi03}. Therefore, for every $t\in \mathbb{R}$:
	\begin{equation*}
		\begin{aligned}
		 0&\le E(e_{p}+te_{r}, e_{q}+te_{s}, e_{p}+te_{r}, e_{q}+te_{s})  \\
		 &= k+ 2(x-z) t^{2},
		\end{aligned}
	\end{equation*}
	whereas
	\begin{equation*}
		\begin{aligned}
		 0&\le E(e_{p}+te_{r}, e_{q}-te_{s}, e_{p}+te_{r}, e_{q}-te_{s})  \\
		 &= k+ 2(z-x) t^{2},
		\end{aligned}
	\end{equation*}
	Since both inequalities hold for all $t$, we obtain $x=z$. A second pair of variations gives
	\begin{equation*}
		\begin{aligned}
		 0&\le E(e_{p}+te_{s}, e_{q}+te_{r}, e_{p}+te_{s}, e_{q}+te_{r}) = k - 2(x+y)t^{2} \\
		 0&\le E(e_{p}+te_{s}, e_{q}-te_{r}, e_{p}+te_{s}, e_{q}-te_{r}) = k + 2(x+y)t^{2} 
		\end{aligned}
	\end{equation*}
	Therefore $x+y=0$. Together with $x=z$ and $x-y+z=0$, we have
	\begin{equation*}
		x=y=z=0.
	\end{equation*}
	Finally, suppose $e_{p}, e_{q}, e_{r}, e_{s}$ belong to four different blocks, we set
	\begin{equation*}
		W= \mathrm{span}\left\{ e_{p}, e_{q}, e_{r}, e_{s} \right\}.
	\end{equation*}
	Every coordinate sectional curvature of $E|_{W}$ is zero, hence the scalar curvature $\sca(E|_{W})=0$. Since $E|_{W}$ has nonnegative sectional curvature, it follows that every sectional curvature of $E|_{W}$ is zero. Since algebraic curvature tensor is determined by its sectional curvatures, so 
	\begin{equation*}
		E|_{W}=0.
	\end{equation*}
	We have now determined all components of $E$, by \eqref{eq:evenRigi02}, $E=kE_{J}/3$ on the sum of any two blocks, while all components involving three or more blocks vanish. The tensor $E_{J}$ has the same vanishing property since $J$ preserves each $\sigma_{a}$. Therefore $E=kE_{J}/3$ holds globally.
\end{proof}

\section{Bochner Vanishing and Endpoint Rigidity}

In this section we combine the algebraic estimate with the Bochner formula. For a Riemannian manifold $(M, g)$, let $I$ be the constant curvature one curvature tensor \eqref{eq:CurvTensI}. At each point $x\in M$, we define 
\begin{equation}\label{eq:notations}
	E_{x}= (\mathrm{Rm})_{x} - K_{\min} I_{x}, \quad \tau (x)=\sca(E_{x})=\sca(M)_{x}-n(n-1)K_{\min}(x).
\end{equation}
Then $E_{x}$ has nonnegative sectional curvature. Define constants:
\begin{equation}\label{eq:constants}
	C_{n}=\frac{3\ell(n-2)}{\ell -1},\qquad \ell = \lfloor\frac{n}{2}\rfloor.
\end{equation}
For later use, we note
\begin{equation*}
	C_{2\ell} = 3n = 6\ell, \qquad C_{2\ell +1} = 3n+ \frac{3}{\ell -1}.
\end{equation*}
The following equavalent sectional-scalar curvature pinching conditions will be used frequently:
\begin{equation*}
	\begin{aligned}
	 K_{\min}\ge \beta_{n}S_{0}&\Longleftrightarrow  \tau \le C_{n} K_{\min} \\
	 K_{\min} > \beta_{n}S_{0}&\Longleftrightarrow  \tau < C_{n} K_{\min} \\
	\end{aligned}
\end{equation*}

Now we can prove our main theorem
\begin{proof}[Proof of \autoref{thm:Main}]
	At every point, $K_{\min}\le S_{0}$, since $S_{0}$ is the average of sectional curvatures; moreover $0<\beta_{n}<1$. Thus the strict condition
	\begin{equation}\label{eq:strict}
		K_{\min}>\beta_{n}S_{0},
	\end{equation}
	forces $S_{0}>0$ and $K_{\min}>0$. Elementary rearrangement show that \eqref{eq:strict} is equivalent to 
	\begin{equation}\label{eq:tau}
		\tau < C_{n}K_{\min}.
	\end{equation}
	By \autoref{thm:SharpEstimate}, at every point $x\in M$,
	\begin{equation*}
		\begin{aligned}
			q_{2}(\mathrm{Rm})&=q_{2}(K_{\min}I+E)\\
					  &\ge \left( 2(n-2)K_{\min}-\frac{2(\ell-1)}{3\ell}\tau \right) \mathrm{Id}\\
					  &=\frac{2(\ell -1)}{3\ell} (C_{n}K_{\min} - \tau	)\mathrm{Id}.
		\end{aligned}
	\end{equation*}
	The right hand side is positive precisely under \eqref{eq:tau}. If $\omega$ is harmonic, integrated Weitzenb\"{o}ck formula gives:
	\begin{equation*}
		0=\int_{M} \left( |\nabla \omega|^{2} + \left\langle q_{2}(\mathrm{Rm})\omega, \omega \right\rangle \right)dV_g
	\end{equation*}
	so $\omega=0$ and $H^{2}(M; \mathbb{R})=0$. If $M$ is orientated, Poincar\'{e} duality gives $H^{n-2}(M; \mathbb{R})=0$. Otherwise apply the same argument to the orientation double cover $\tilde{M}$. The pullback $H^{n-2}(M; \mathbb{R})\to H^{n-2}(\tilde{M}; \mathbb{R})$ is injective, so the same vanishing holds for $M$.
\end{proof}

\begin{proof}[Proof of \autoref{thm:WeakEnd}]
	We divide the proof into five steps in order to separate the Bochner formular, holonomy, equality discussion and uniformization argument.
	
	\emph{Claim 1: Harmonic two-forms are parallel.}

	Clearly the weak pinching condition implies $S_{0}\ge 0$, $K_{\min}\ge 0$, and it is equivalent to 
	\begin{equation}\label{eq:tauE}
		\tau\le C_{n}K_{\min}.
	\end{equation}
	It follows that $q_{2}(\mathrm{Rm})\ge 0$. For a harmonic two-form $\omega$, the integrated Weitzenb\"{o}ck identity has a nonnegative integrand. It follows that
	\begin{equation*}
		\nabla \omega =0, \qquad q_{2}(\mathrm{Rm})\omega =0.
	\end{equation*}
	Assume from now on, that $g$ is not flat, and fix a point $p\in M$. 

	\emph{Claim 2: The restricted holonomy representation $\operatorname{Hol}_{p}^{0}(M) \curvearrowright T_{p}M$ is irreducible. }

	Suppose the representation were reducible. Since $M$ is complete, by the de Rham decomposition theorem, the universal cover of $M$ splits as a nontrivial Riemannian product, cf \cite{deR1952}, \cite{Bes1987}. At each point mixed planes have zero sectional curvature, it follows that $K_{\min}=0$ at every point. Then the pinching condition gives $S_{0}\equiv 0$. Since sectional curvatures are nonnegative, it forces $g$ to be flat, a contradiction. Therefore, the restricted holonomy representation is irreducible.

	\emph{Claim 3: The bounds on $b_{2}$.}

	Let $\omega$ be a parallel two-form, and let $B_{\omega}$ be the associated parallel skew-adjoint endomorphism defined by:
	\begin{equation*}
		\omega(X, Y)= \left\langle B_{\omega}X, Y \right\rangle.
	\end{equation*}
	At point $p$, $B_{\omega}$ commutes with the restricted holonomy. Therefore both $\ker B$ and $\operatorname{Im}B$ are invariant. So irreducibility makes every nonzero such $B_{\omega}$ invertible. This is impossible in odd dimension; hence $b_{2}(M)=0$ whenever $n$ is odd. For even $n$, set
	\begin{equation*}
		\mathcal{D}_{p}:=\operatorname{End}_{\operatorname{Hol}^{0}_{p}(M)}(T_{p}M)=
		\left\{ B\in \operatorname{End}_{\mathbb{R}}(T_{p}M)\ : \ Bh=hB \ \forall h\in \operatorname{Hol}_{p}^{0}(M) \right\}.
	\end{equation*}
	Real Schur theory and Forbenius' theorem gives $\mathcal{D}_{p}\cong \mathbb{R}, \mathbb{C}$ or $\mathbb{H}$.
	Then $B_{\omega}^{*}=-B_{\omega}$ and $\nabla g = \nabla \omega =0$ implies $\nabla B_{\omega} =0$. Therefore $B_{\omega}|_{p}$ commutes with the full holonomy group. It follows that its kernel and image are invariant, so irreducibility makes every nonzero $B_{\omega}$ invertible. This is impossible in odd dimension, therefore $b_{2}(M)=0$ when $n$ is odd.
	For even $n$, we set
	\begin{equation*}
		\mathcal{D}_{p} = \mathrm{End}_{\mathrm{Hol}_{p}^{0}}(T_{p}M).
	\end{equation*}
	By real Schur theory and Frobenius' theorem, $\mathcal{D}_{p}$ is isomorphic to $\mathbb{R}, \mathbb{C}$, or $\mathbb{H}$; the metric adjoint induces the standard conjugation, so $\mathcal{D}_{p}$ is a finite dimensional real $C^{*}$-division algebra; its involution is indentified with standard conjugation on $\mathbb{R}, \mathbb{C}$ or $\mathbb{H}$. The skew-adjoint parts have real dimension $0,1$ and $3$ respectively cf \cite[Chapter 10]{Bes1987}. Evaluation map
	\begin{equation*}
		\begin{aligned}
			\mathrm{ev}_{p}: \{ \text{Parallel two-forms}\} &\to \mathcal{D}_{p}\cap \mathfrak{so}(T_{p}M).\\
			\omega ^{}\mapsto B_{\omega,p}.
		\end{aligned}
	\end{equation*}
	injects the space of parallel two-forms into $\mathcal{D}_{p}\cap \mathfrak{so}(T_{p}M)$. 

	The quaternionic case would give three parallel complex structures on the universal cover, making it hyperk\"{a}hler and hence Ricci-flat, \cite{Joy2000}. Since our sectional curvature is nonnegative, Ricci flat forces every sectional curvature to vanish, a contradiction. Therefore only real and complex cases remain, and Hodge theory gives:
	\begin{equation*}
		b_{2}(M)\le 1.
	\end{equation*}

	\emph{Claim 4: pointwise equality on the nonflat locus.} 

	Now we assume $n=2\ell$ and $b_{2}(M)>0$. Choose a harmonic two-form $\omega$ normalized to have unit length. It is parallel by the Claim 1, and Claim 2 shows that its associated endomorphism $B_{\omega}$ is invertible; hence $\omega$ must have full rank. Let
	\begin{equation*}
		U = \left\{ x\in M\ : \ \mathrm{Rm}_{x}\ne 0 \right\}.
	\end{equation*}
	This is a nonempty open subset of $M$. At every point of $x\in U$, nonnegative sectional curvature and $\mathrm{Rm}_{x}\ne 0$ give $S_{0}>0$, and the pinching assumption give $K_{\min}>0$. Since $C_{n}=3n$ for $n$ even. \autoref{thm:SharpEstimate} and $\tau\le 3nK_{\min}$ gives the following:
	\begin{equation*}
		\begin{aligned}
		0 &= \left\langle q_{2}(\mathrm{Rm})_{x}\omega_{x}, \omega_{x} \right\rangle  \\
	 	&\ge 2(n-2) K_{\min}(x) - \frac{2(\ell-1)}{3\ell} \tau(x)\\
	 	&\ge 2(n-2) K_{\min}(x) - \frac{2(\ell -1)}{3\ell} 3nK_{\min}(x) =0
		\end{aligned}
	\end{equation*}
	Therefore 
	\begin{equation}\label{eq:Endpoint03}
	\tau(x) = 3n K_{\min}(x), 
	\end{equation}
	and $(E_{x}, \omega_{x})$ realizes equality in \autoref{thm:SharpEstimate} at every $x\in U$.
	
	Define the parallel skew-adjoint field
	\begin{equation*}
		J= \sqrt{\ell} B_{\omega}.
	\end{equation*}
	By \autoref{lem:Even}, at every point $x\in U$,
	\begin{equation}\label{eq:Endpoint02}
		(B_{\omega})_{x}^{2} = - \ell^{-1} \mathrm{Id}_{T_{x}M},\qquad E_{x} = \mu(x) (E_{g,J})_{x}
	\end{equation}
	for some positive function $\mu$ defined on $U$. The parallel endomorphism field $B_{\omega}^{2}+\ell^{-1} \mathrm{Id}_{TM}$ vanishes at the nonempty set $U$, hence on the whole $M$. Therefore we just proved $J^{2} = - \mathrm{Id}_{TM}$. Moreover, $J^{*} = -J$ and $\nabla J=0$. It follows that $J$ is a parallel orthogonal complex structure and $(M, g, J)$ is K\"{a}hler.

	By \eqref{eq:JPro01}, $\sca((E_{J}))=3n$. Comparing scalar curvature in \eqref{eq:Endpoint02} and using \eqref{eq:Endpoint03} gives $\mu = K_{\min}$ on $U$. Therefore, we have
	\begin{equation}\label{eq:Endpoint01}
		\mathrm{Rm} = K_{\min} (I + E_{J})\qquad \text{on}\ U.
	\end{equation}
	
	\emph{Claim 5: The metric is Fubini-Study upto rescaling.}

	Since $\operatorname{Ric}(I_{x}) = (n-1)g_{x}$ and $\operatorname{Ric}(E_{g, J})_{x} = 3 g_{x}$, contraction of \eqref{eq:Endpoint01} gives
	\begin{equation}\label{eq:Endpoint04}
		\operatorname{Ric}_{g} = (n+2) K_{\min} g,\qquad \sca_{g}= n(n+2)K_{\min}
		\qquad \text{on}\ U.
	\end{equation}
	In particular $K_{\min} = \sca_{g}/ n(n+2)$ is smooth on $U$. The contracted Bianchi identity gives
	\begin{equation*}
		(n+2)d(K_{\min}) = \operatorname{div}(\operatorname{Ric}_{g}) = \frac{1}{2} d \sca_{g} = \frac{n(n+2)}{2}d K_{min}
	\end{equation*}
	Since $n\ge 4$, it follows that $d K_{\min} =0$ on $U$. So $K_{\min}$ is constant on each connected component of $U$. Fix a connected component $U_{0}$ of $U$, and write $K_{\min}\equiv \alpha_{0} >0$ on $U_{0}$. Let $p_{i}\in U_{0}$ be a sequence of points converging to $p$, taking limit in \eqref{eq:Endpoint01} gives
	\begin{equation*}
		\mathrm{Rm}_{p} = \alpha_{0}\left( I_{p} + (E_{g, J})_{p} \right) \ne 0,
	\end{equation*}
	Therefore $p\in U$. Since $U_0$ is a connected component of $U$, it is closed in $U$, therefore $p\in U_0$. Hence $U_0$ is a closed. Since $M$ is assumed to be connected, $U_0=M$. Consequently, $M$ is a complete K\"{a}hler manifold of constant holomorphic sectional curvature $4K_{\min}$, it is follows that $M$ is isometric to a scaled Fubini-Study $\mathbb{CP}^{\ell}$ by \cite{Haw1953}. 
\end{proof}

\begin{proof}[Proof of \autoref{cor:Berger}]
	For $n=5$, $\beta_{5}=10/19$. The strict hypothesis implies $S_{0}>0$ and $K_{\min}>0$ as in the proof of \autoref{thm:Main}. Therefore $M$ has positive sectional curvature. By Synge's theorem $M^{5}$ is orientable \cite{Syn1936} therefore $b_0=b_5=1$. The Bochner formula for one-forms and $\operatorname{Ric}>0$ gives $b_{1}(M)=0$. \autoref{thm:Main} gives $b_{2}=b_{3}=0$, and Poincar\'{e} duality gives $b_{4}=b_1=0$. Therefore $M$ is a rational homology sphere.
\end{proof}	

\printbibliography
\end{document}